\documentclass[10pt,conference]{ieeeconf}
\usepackage{textcomp}
\usepackage{url}            % simple URL typesetting
\usepackage{amsfonts}       % blackboard math symbols
\usepackage{xcolor}

\usepackage{amsmath, amssymb, mathtools, amsthm}
\usepackage{graphicx}
\usepackage{algorithm, algorithmic}
\usepackage{bm}
\usepackage{float}
\usepackage{color}
\usepackage{epstopdf}

\makeatletter
\let\NAT@parse\undefined
\makeatother

\usepackage{hyperref}
\makeatletter
\hypersetup{colorlinks=true}
\AtBeginDocument{\@ifpackageloaded{hyperref}
  {\def\@linkcolor{blue}
  \def\@anchorcolor{red}
  \def\@citecolor{red}
  \def\@filecolor{red}
  \def\@urlcolor{red}
  \def\@menucolor{red}
  \def\@pagecolor{red}
\begingroup
  \@makeother\`%
  \@makeother\=%
  \edef\x{%
    \edef\noexpand\x{%
      \endgroup
      \noexpand\toks@{%
        \catcode 96=\noexpand\the\catcode`\noexpand\`\relax
        \catcode 61=\noexpand\the\catcode`\noexpand\=\relax
      }%
    }%
    \noexpand\x
  }%
\x
\@makeother\`
\@makeother\=
}{}}
\makeatother

\newtheorem{Theorem}{Theorem}

\newtheorem{Lemma}{Lemma}

\newtheorem{Problem}{Problem}

\newtheorem*{Problem*}{Problem}

\newtheorem{Remark}{Remark}

\newtheorem{Assumption}{Assumption}

\newtheorem{Definition}{Definition}

\renewcommand{\baselinestretch}{0.9}

\IEEEoverridecommandlockouts

\def\BibTeX{{\rm B\kern-.05em{\sc i\kern-.025em b}\kern-.08em
    T\kern-.1667em\lower.7ex\hbox{E}\kern-.125emX}}

\begin{document}

\title{\LARGE{\bf Robust Control Barrier and Control Lyapunov Functions with Fixed-Time Convergence Guarantees}}

\author{Kunal~Garg \and
% \IEEEmembership{Student Member, IEEE}, 
Dimitra~Panagou
% \IEEEmembership{Senior Member, IEEE}
\thanks{
% This paragraph of the first footnote will contain the date on which you submitted your brief for review. It will also contain support information, including sponsor and financial support acknowledgment. For  example, ``This work was supported in part by the U.S. Department of  Commerce under Grant BS123456.'' 
The authors would like to acknowledge the support of the Air Force Office of Scientific Research under award number FA9550-17-1-0284.}
\thanks{The authors are with the Department of Aerospace Engineering, University of Michigan, Ann Arbor, MI, USA; \texttt{\{kgarg, dpanagou\}@umich.edu}.}
}
\maketitle
\begin{abstract}
This paper studies control synthesis for a general class of nonlinear, control-affine dynamical systems under additive disturbances and state-estimation errors. We enforce forward invariance of static and dynamic safe sets and convergence to a given goal set within a user-defined time in the presence of input constraints. We use robust variants of control barrier functions (CBF) and fixed-time control Lyapunov functions (FxT-CLF) to incorporate a class of additive disturbances in the system dynamics, and state-estimation errors. To solve the underlying constrained control problem, we formulate a quadratic program and use the proposed robust CBF-FxT-CLF conditions to compute the control input. We showcase the efficacy of the proposed method on a numerical case study involving multiple underactuated marine vehicles.
\end{abstract}

% \begin{IEEEkeywords}
% Distributed algorithms, Optimization, Nonlinear control systems,  Power generation dispatch
% \end{IEEEkeywords}

\section{Introduction}
With the advent of complex missions that require multi-robot systems to execute various tasks in parallel, the need for a systematic synthesis of algorithms that enable the underlying objectives has emerged. 
% Signal Temporal Logic (STL) is one such language that helps encode multi-task problems in a concise mathematical expression. 
Standard objectives in such missions include, but are not limited to, requiring each robot to stay within a given subset of the state space for a given time duration, while keeping a point of interest in its field of view, and reaching a destination within a given time horizon. It is also important that each robot always maintains a safe distance from stationary and moving objects or other robots in the environment. In problems where the objective is to stabilize the closed-loop trajectories to a given desired point or a set, control Lyapunov functions (CLFs) are very commonly used to design the control input \cite{romdlony2016stabilization,ames2017control}. Temporal constraints, i.e., constraints pertaining to convergence within a {fixed} time, appear in time-critical applications, for instance when a task must be completed within a given time interval. The use of fixed-time stability (FxTS) \cite{polyakov2012nonlinear} has enabled the synthesis of controllers guaranteeing finite- or fixed-time reachability to the desired point or a set \cite{garg2019control}.  Similarly, safety or containment of the closed-loop trajectories in a subset of the state-space can be enforced using control barrier functions (CBFs). Traditionally, CBFs have been used to encode safety with respect to static safe sets arising due to the presence of stationary obstacles or unsafe regions in the state-space (see \cite{romdlony2016stabilization,wieland2007constructive}) and with respect to dynamically-changing safe sets, such as in multi-agent systems ( \cite{li2018formally,lindemann2019control}).

The development of fast optimization solvers has enabled the online control synthesis using quadratic programs (QPs), where CLF and CBF conditions are encoded as linear constraints, while the objective is to minimize the norm of the control input \cite{ames2017control,glotfelter2017nonsmooth} or the deviation of the control input from a \textit{nominal} controller \cite{li2018formally,wang2017safety}. Most of the prior work on QP-based control design enforces the safety constraint with one fix CBF condition and uses a slack term in the CLF condition to guarantee that the QP is feasible in the absence of input constraints. However, control input constraints should be also considered in the design step, otherwise, the derived control input might not be realizable due to actuator limits, and might lead to violation of the safety requirements. In the prior work \cite{garg2019prescribedTAC}, we considered an ideal case without any disturbances, and proposed a QP with feasibility guarantees that achieve forward invariance of a safe set and reachability to a goal set, even in the presence of control input constraints. 

% The dynamical systems used to model the mission and the interactions of the robots with their environment are often subject to model uncertainties and disturbances. 
Encoding safety in the presence of disturbances can be done using robust CBFs \cite{jankovic2018robust,cortez2019control}. While the aforementioned work considers bounded additive disturbance in the system dynamics, it is generally assumed that the system states are available without any errors. In their majority, earlier work in the literature on multi-agent collision avoidance using CBFs \cite{chen2015decoupled,rodriguez2014trajectory, glotfelter2017nonsmooth} assumes perfect knowledge of the states of the agents and no state-estimation errors. In this paper, apart from additive disturbances in the system dynamics, we also consider bounded state-estimations errors and incorporate them in the robust CBF design to guarantee forward invariance of the safe sets.

This paper studies QP-based control synthesis for multi-task problems involving agents of nonlinear, control-affine dynamics, with the following objectives for the closed-loop trajectories: (i) remain inside a static safe set, (ii) remain inside a time-varying safe set (arising for instance due to the presence of moving obstacles or neighboring agents), and (iii) reach a given goal set within a user-defined time. We first present robust CBF conditions to guarantee forward invariance while incorporating both the disturbance in the system dynamics as well as the state-estimation error. Then, utlizing the fixed-time stability conditions from \cite{garg2021FxTSDomain}, we propose a robust fixed-time CLF condition to guarantee convergence to the desired goal set within the user-defined time, extending the prior results in \cite{li2018formally,garg2019control,garg2019prescribedTAC}. Finally, we merge the presented robust CBF-FxT-CLF conditions in a QP formulation, show its feasibility, and discuss the conditions under which the control input defined as the solution of the QP solves the multi-task problem. We showcase the efficacy of the proposed method via a multi-agent case study involving under-actuated marine vehicles.
% with higher-order relative degree CBFs.
 
\section{Mathematical Preliminaries}\label{sec: math prelim}
% \subsection{Notation and Definition}
%We denote by $\dot x(t) = \frac{\text{d}x(t)}{\dt}$ the time derivative of function $x(t)$ and by $y'(s) = \frac{\text{d}y(s)}{\ds}$ the derivative of $y(s)$ with respect to the stretched time $s$.
% We use a standard notation throughout this paper. Specifically,
% \subsection{Notation}
\noindent \textbf{Notations}: In the rest of the paper, $\mathbb R$ denotes the set of real numbers and $\mathbb R_+$ denotes the set of non-negative real numbers. We use $\|x\|$ to denote the Euclidean norm of a vector $x\in \mathbb R^n$. $|x|$ denotes the absolute value when $x\in \mathbb R$, and cardinality, or the number of elements, when $x\in 2^N$ is a set, for some positive integer $N$. We use $\partial S$ to denote the boundary of a closed set $S\subset \mathbb R^n$ and $\textrm{int}(S)$ to denote its interior, and $\|x\|_S = \inf_{y\in S}\|x-y\|$, to denote the distance of $x\in \mathbb R^n$ from the set $S$. We use $\mathbb B_\epsilon$ to denote a ball of radius $\epsilon>0$ centered at the origin. 

% For two sets $A, B$, we use $A-B = \{z\; |\; z = x-y, x\in A, y\in B\}$ to denote the Minkowski difference of the sets. In particular, $S-\mathbb B_\epsilon$ denote the shrinkage of the set $S$ by $\epsilon$.  
\vspace{2pt}
% \subsection{Problem formulation}
\noindent \textbf{System model}: In this work, we consider a multi-task problem for the dynamical system given as:
\begin{align}\label{eq: NL pert cont affine}
    \dot x(t) & = f(x(t)) + g(x(t))u + d(t,x),
\end{align}
where $x\in \mathbb R^n, u \in \mathcal U\subset \mathbb R^m$ are the state and the control input vectors, respectively, with $\mathcal U$ the control input constraint set, $f:\mathbb R^n\rightarrow\mathbb R^n$ and $g: \mathbb R^n\rightarrow\mathbb R^{n\times m}$ are continuous functions and $d:\mathbb R_+\times \mathbb R^n\rightarrow \mathbb R^n$ is an unknown additive disturbance. We make the following assumption. 

\begin{Assumption}\label{assum d}
There exists $\gamma>0$ such that for all $t\geq 0$ and $x\in \mathcal D\subset \mathbb R^n$, $ \|d(t,x)\|\leq \gamma$.
\end{Assumption}

\noindent Assumption \ref{assum d} implies that the disturbance $d$ is uniformly bounded in the domain $\mathcal D$. This is a standard model to account for various types of uncertainties, environmental noises, and external disturbances (see, e.g., \cite{cortez2019control}). Furthermore, we assume that the state $x$ is not perfectly known, to account for sensor noises and uncertainties. More specifically, we consider that only an estimate of the system state, denoted as $\hat x$, is available, that satisfies:
\begin{align}\label{eq: NL pert hat x}
    \dot{\hat x} = f(\hat x) + g(\hat x)u ,+ d(t, \hat x),
\end{align}
and make the following assumption.  

\begin{Assumption}\label{assum hat x eps}
There exists an $\epsilon>0$ such that $\|\hat x(t)-x(t)\|\leq \epsilon$, for all $t\geq 0$. 
\end{Assumption}

We now define some notations and functions necessary to state the main problem. Let $h_{S}:\mathbb R^n\rightarrow\mathbb R$ be a continuously differentiable function defining the \textit{static} safe set $S_{S} = \{x\; |\; h_{S}(x)\leq 0\}$. The system trajectories might also need to maintain safety with respect to a dynamically-changing safe set, for instance due to the presence of moving obstacles or other agents in a multi-agent scenario. In such a case, a centralized collision avoidance scheme would require each agent $i$ to be in a safe set defined as $\{x_i(t)\; |\; h(x_i(t),x_j(t)) \leq 0\}$ for all $j\neq i$, where $x_i, x_j\in \mathbb R^n$ are the states of agent $i$ and $j$ (see Section \ref{sec: simulations} for a multi-agent case study). In particular, if $x_i$ represents the position of the agent $i$, then the function $h$ can be chosen as $h(x_i(t),x_j(t)) = d_s^2-\|x_i(t)-x_j(t)\|^2$, where $d_s>0$ is the safety distance. In this case, we can define $h_T(t,x_i) = \max_{j\neq i}h(x_i(t),x_j(t))$ so that it encodes safety with respect to all other agents.\footnote{One can use a smooth approximation for the $\max$ function, e.g., $h_T = \log(\sum_j e^{h_{ij}})$, so that the resulting function $h_T$ is continuously differentiable (see \cite{lindemann2019control}).} To encode safety with respect to a general time-varying safe sets, such as the one discussed above, let $h_{T} :\mathbb R_+\times\mathbb R^n\rightarrow \mathbb R$ be a continuously differentiable function defining the \textit{time-varying} safe set $S_{T}(t) = \{x\; |\; h_{T}(t,x)\leq 0\}$. Finally, let $h_{G}:\mathbb R^n\rightarrow\mathbb R$ a continuously differentiable function defining the goal set $S_{G} = \{x\; |\; h_{G}(x)\leq 0\}$. 
% We assume that the interior of the safe sets, $\textrm{int}(S_{S}) = \{x\; |\; h_{S}(x)<0\}$, is non-empty. 
The problem formulation follows.

\begin{Problem}\label{Problem: multiagent ST}
Find a control input $u(t)\in \mathcal U =\{v\in \mathbb R^m \; |\; u_{j,min}\leq v\leq u_{j,max}, j = 1, 2, \ldots, m\}$, $t\geq 0$, and a set $D$, such that for all $x(0)\in D\subset S_{S}\cap S_{T}(0)$, the closed-loop trajectories of \eqref{eq: NL pert cont affine} satisfy
\begin{itemize}
    \item[(i)] 
    % Convergence to the goal set within user-defined time: 
    $x(\bar T)\in S_{G}$ for some user-defined $\bar T>0$;
    \item[(ii)] 
    % Forward invariance of the static safe set: 
    $x(t)\in S_S$ for all $t\geq 0$;
    \item[(iii)]
    % Forward invariance of the time-varying safe set:
    $x(t)\in S_{T}$ for all $t\geq 0$.
\end{itemize}
\end{Problem}

\noindent Here, $\mathcal U$ is box-constraint set where $u_{j,min}< u_{j,max}$ are the lower and upper bounds on the individual control input $v_j$ for $j = 1, 2, \ldots, m$, respectively. Input constraints of this form are very commonly considered in the literature \cite{ames2017control,cortez2019control}. We can write $\mathcal U$ in a compact form as $\mathcal U = \{v\; |\; A_{u}v\leq b_{u}\}$ where $A_u\in \mathbb R^{2m\times m}, b_u\in \mathbb R^{2m}$.
% where {\small
% \begin{align*}
%  A_u =\begin{bmatrix}1 & 0 & \cdots & 0\\ -1 & 0 & \cdots & 0 \\ 0 & 1 & \cdots & 0\\ 0 & -1 & \cdots & 0\\ & & \vdots & \end{bmatrix} \in \mathbb R^{2m\times m},   b_u = \begin{bmatrix}u_{1,max}\\ -u_{1,min}\\ u_{2,max}\\ -u_{2,min}\\ \vdots\end{bmatrix}\in \mathbb R^{2m}.  
% \end{align*}}\normalsize

% \noindent Thus, we seek a feedback control input $u(t) = u(\hat x(t))$ that solves Problem \ref{Problem: multiagent ST}. 

% {\color{blue}Since the actual states $x$ is not available, it is important to establish relationship between set membership of the actual state $x$ and the estimated state $\hat x$ for any given set $S\subset \mathbb R^n$. Under Assumption \ref{assum hat x eps}, it can be easily shown that for any $t\geq 0$,
% \begin{itemize}
%     \item(i) $x(t)\in S-\mathbb B_\epsilon \implies \hat x(t)\in S$;
%     \item(ii) $\hat x(t)\in S-\mathbb B_\epsilon \implies x(t)\in S$.
% \end{itemize}
% Property (i) will be used to relate the set of estimated states $\{\hat x(0)\}$ with that of the actual states $\{x(0)\}$ (e.g., for computing the admissible set of initial conditions for safety) and property (ii) will be used to relate the set $\{x(t)\}$ with $\{\hat x(t)\}$ at any future time (e.g., to establish the convergence of the system to the goal set). 
% }

% \subsection{Forward invariance of a set}
\vspace{2pt}
\noindent \textbf{Forward invariance}: We first review a sufficient condition for guaranteeing forward invariance of a set in the absence of the disturbances and noises. Define $S(t) = \{x\; |\; h(t,x)\leq 0\}$ for some continuously differentiable $h:\mathbb R_+\times\mathbb R^n\rightarrow\mathbb R$.
% \begin{Lemma}\label{lemma: nec suff safety}
% Let $d \equiv 0$ and the solution $x(t)$ of \eqref{eq: NL pert cont affine} exist and be unique in forward time. Then, the set $S(t) = \{x\; |\; h(t,x)\leq 0\}$ is forward invariant for the trajectories of \eqref{eq: NL pert cont affine} for all $x(0)\in S(0)$ if the following condition holds:
% \begin{equation}\label{eq: safety cond}
%     \inf_{u\in \mathcal U}\left\{L_{f}h(t,x)+L_{g}h(t,x)u+\frac{\partial h}{\partial t}(t,x)\right\}\leq 0,
% \end{equation}
% for all $x(t)\in \partial S(t)$, $t\geq 0$, where $\partial S(t) \triangleq \{x\; |\; h(t,x) = 0\}$ is the boundary of the safe set $S(t)$. 
% \end{Lemma}

\begin{Lemma}\label{lemma: nec suff safety}
Let $d \equiv 0$ and the solution $x(t)$ of \eqref{eq: NL pert cont affine} exist and be unique in forward time. Then, the set $S(t) = \{x\; |\; h(t,x)\leq 0\}$ is forward invariant for the trajectories of \eqref{eq: NL pert cont affine} for all $x(0)\in S(0)$ if the following condition holds:
\begin{equation}\label{eq: safety cond}
    \inf_{u\in \mathcal U}\left\{L_{f}h(t,x)+L_{g}h(t,x)u+\frac{\partial h}{\partial t}(t,x)\right\}\leq \alpha(-h(t,x)),
\end{equation}
for all $x\in S(t)$, $t\geq 0$ where $\alpha$ is a locally Lipschitz class-$\mathcal K$. Furthermore, if $h(0,x(0))<0$, then for any $T\geq 0$, $h(t,x(t))<0$ for all $0\leq t\leq T$. 
\end{Lemma}

% We also need the following Lemma to prove our main results. 

% We make the following assumption so that Problem \ref{Problem: multiagent ST} is feasible.
% \begin{Assumption}\label{Assum feas safe set}
% The following condition holds
% \begin{equation}\label{eq: assum safety cond Ss}
%   { \inf_{u\in \mathcal U}}\{L_fh_s(x)+L_gh_s(x)u\}\leq 0, \quad \forall x\in \partial S_{S}.
% \end{equation}
% Furthermore, $S_g\bigcap S_s\neq \emptyset$. 
% \end{Assumption}

% {Similar assumptions have been used in literature, either explicitly (see e.g. \cite{romdlony2016stabilization}) or implicitly (see e.g. \cite{ames2017control}).}

% \begin{Lemma}\label{lemma forw var int set}
% Let $d \equiv 0$ and the solution $x(t)$ of \eqref{eq: NL pert cont affine} exist and be unique in forward time. If $h(0,x(0))<0$, and the inequality {\small
% \begin{align}\label{eq h dot class K func}
%  \inf_{u\in \mathcal U}\dot h & = \inf_{u\in \mathcal U}\left\{L_{f}h(t,x(t))+L_{g}h(t,x(t))u + \frac{\partial h}{\partial t}(t,x(t))\right\}\nonumber\\
%  & \leq \alpha(- h(t,x(t)))   
% \end{align}}\normalsize
% holds for some locally Lipschitz class-$\mathcal K$\footnote{Recall that a function $\alpha\in \mathcal K$ if it is continuous, strictly increasing, and $\alpha(0) = 0$.} function $\alpha$ for all $t\geq 0$, then for any $T\geq 0$, $h(t,x(t))<0$ for all $0\leq t\leq T$ and $h(t,x(t))\leq 0$ for all $t\geq 0$. 
% \end{Lemma}

\noindent A function that satisfies \eqref{eq: safety cond} is called a valid CBF by the authors in \cite{lindemann2019control}, and a zeroing-CBF by the authors in \cite{ames2017control}. 
% We make the following assumption so that Problem \ref{Problem: multiagent ST} is feasible.
% \begin{Assumption}\label{Assum feas}
% The functions $h_{S}, h_{T}(t)$ satisfy \eqref{eq: safety cond} on the boundary of the sets $S_S$ and $S_T(t)$ for all $t\geq 0$, respectively, for \eqref{eq: NL pert cont affine} with $d \equiv 0$.
% % the following condition holds
% % \begin{equation}\label{eq: assum safety cond}
% %   \inf_{u\in \mathcal U}\{L_{f}h_{S}(x)+L_{g}h_{S}(x)u\}\leq 0, 
% % \end{equation}
% % for all $x\in \partial S_{S}$. 
% % Furthermore, it holds that 
% % \begin{equation}\label{eq: assum safety cond}
% %   \inf_{u\in \mathcal U}\{L_fh(x)+L_gh(x)u\}\leq 0,
% % \end{equation}
% % for all $i = 1, 2, \dots, N$, and 
% \end{Assumption}
% % \textbf{Condition (3) is the same as (2) so do not repeat}
% \noindent Similar assumptions have been used in literature either explicitly (see e.g., \cite{romdlony2016stabilization}) or implicitly (see e.g., \cite{ames2017control}). 

% \subsection{Fixed-time convergence}
\vspace{2pt}
\noindent \textbf{Fixed-time stability}: Next, we review a sufficient condition for fixed-time stability of the origin for the closed-loop trajectories of \eqref{eq: NL pert cont affine}. 

\begin{Lemma}[\hspace{-0.3pt}\cite{garg2021FxTSDomain}]\label{Th: FxTS new}
Let $V:\mathbb R^n\rightarrow \mathbb R$ be a continuously differentiable, positive definite, proper function, satisfying{\small
\begin{align}
    \dot V(x(t)) \leq -c_1(V(x(t)))^{a_1}-c_2(V(x(t)))^{a_2}+c_3V(x(t)),
\end{align}}\normalsize
with $c_1, c_2>0$, $c_3\in \mathbb R$, $a_1 = 1+\frac{1}{\mu}$, $a_2 = 1-\frac{1}{\mu}$ for some $\mu>1$, along the closed-loop trajectories of \eqref{eq: NL pert cont affine} under a continuous control input $u(x)$. Then, there exists a neighborhood $D$ of the origin such that for all $x(0)\in D $, the trajectories of \eqref{eq: NL pert cont affine} reach the origin in a fixed time $T$ where $T, D$ are known functions of $\mu$ and $\frac{c_3}{2\sqrt{c_1c_2}}$.
\end{Lemma}

\noindent In this work, without loss of generality, it is assumed that the functions $h_{S}, h_{T}, h_G$ are relative-degree one functions. For higher-relative degree functions, higher-order CLF and CBF conditions can be used. For example, if the function $h_{S}$ is of relative degree 2 (as in the case study presented in Section \ref{sec: simulations}), then following the results in \cite{nguyen2016exponential}, it can be shown that satisfaction of the inequality 
{{
\begin{align}\label{eq: h rel deg 2}
    \hspace{-5pt}L^{2}_fh_{S} +L_gL_fh_{S}u + 2L_fh_{S} + h_{S}\leq \alpha (-L_fh_{S}-h_{S}), 
\end{align}}}\normalsize
for some $\alpha\in \mathcal K$ implies that the set $\bar S_{S}  = \{x\; |\; L_fh_{S}(x)+h_{S}(x)\leq 0\}\subset S_{S}$ is forward-invariant. In this case, one can define $\bar h_{S} = L_fh_{S}+h_{S}$ so that \eqref{eq: h rel deg 2} reads $L_f\bar h_{S} + L_g\bar h_{S}u\leq \alpha(-\bar h_{S})$, which is same as \eqref{eq: safety cond}, thus guaranteeing forward invariance of the set $\bar S_{S}$. Interested reader is referred to \cite{xiao2019control} for more details on higher-order CBF conditions.

\section{Main results}
% \subsection{Robust CBF and CLF}
\noindent \textbf{Robust CBF and CLF}: First, we present conditions for robust CBFs so that the safety requirements (ii) and (iii) in Problem \ref{Problem: multiagent ST} can be satisfied in the presence of the disturbance $d$ and error $\epsilon$. We make the following assumption. 

\begin{Assumption}\label{assum hs ht L}
There exist $l_S, l_G, l_T>0$ such that $\left \|\frac{\partial h_{S}}{\partial x}(x)\right\| \leq l_S, \left \|\frac{\partial h_{G}}{\partial x}(x)\right\| \leq l_G \left \|\frac{\partial h_{T}}{\partial x}(t,x)\right\| \leq l_T$,
for all $x\in \mathcal D\subset \mathbb R^n$, and all $t\geq 0$.
\end{Assumption}

\noindent Since the functions $h_{S}, h_{T}$ are continuously differentiable, Assumption \ref{assum hs ht L} can be easily satisfied in any compact domain $\mathcal D$. Corresponding to the set $S(t) = \{x\; |\; h(t,x)\leq 0\}$ for some continuously differentiable $h:\mathbb R_+\times\mathbb R^n\rightarrow\mathbb R$, define $\hat S_\epsilon(t) = \{\hat x\; |\; h(t,\hat x)\leq -l\epsilon\}$, where $l = \sup \|\frac{\partial h(t,x)}{\partial x}\|$ is the Lipschitz constant of the function $h$. We define the notion of a robust CBF as follows.

\begin{Definition}[\textbf{Robust CBF-$S$}]
A continuously differentiable function $h:\mathbb R_+\times\mathbb R^n\rightarrow\mathbb R$ is called a robust CBF for the set $S(t) = \{x\; |\; h(t,x)\leq 0\}$ for \eqref{eq: NL pert cont affine} w.r.t. a disturbance $d$ satisfying Assumption \ref{assum d} if the following condition holds
\begin{align}\label{eq: robust safety cond}
\begin{split}
    &\inf_{u\in \mathcal U}\left\{L_{f}h(t,x(t))+L_{g}h(t,x(t))u+\frac{\partial h}{\partial t}(t,x(t))\right\}\\ 
    & \quad \quad \leq \alpha(-h(t, x(t)))-l\gamma,
\end{split}
\end{align}
for some locally Lipschitz class-$\mathcal K$ function $\alpha$ and for all $x(t)\in S(t)$, $t\geq 0$.
\end{Definition}

\noindent Note that we use the worst-case bound of $\|\frac{\partial h}{\partial x}d\| = l\gamma$ to define the robust CBF. This condition can be relaxed if more information than just the upper bound of the disturbance is known. We can now state the following lemma that relates the robust CBF condition with forward invariance of the set $S(t)$ in the presence of the disturbance $d$. For any function $\phi:\mathbb R_+\times \mathbb R^n\rightarrow\mathbb R$ with Lipschitz constant $l_\phi$, define 
\begin{align}\label{eq: hat func def}
\hat \phi(t,\cdot) = \phi(t,\cdot) + l_\phi\epsilon.
\end{align}

\begin{Lemma}\label{lemma: robust CBF}
Let the solution $x(t)$ of \eqref{eq: NL pert cont affine} exist and be unique in forward time, and $\hat h$ be a robust CBF-$S$ for \eqref{eq: NL pert hat x}. Then there exists a control input $u\in \mathcal U$ such that the set $S(t)$ is forward invariant for the trajectories of \eqref{eq: NL pert cont affine} for all $\hat x(0)\in \hat S_\epsilon(0)$ .
\end{Lemma}
\begin{proof}
% First, note that $\hat x(0)\in \hat S_\epsilon(0)-\mathbb B_{2\epsilon}$ implies that $ x(0)\in \hat S_\epsilon(0)-\mathbb B_\epsilon$, which in turn guarantees that . 
Using the mean value theorem, we have that there exists $z\in \mathbb R^n$ such that
\begin{align*}
    h(t,x) & = h(t,\hat x + (x-\hat x)) = h(t,\hat x) + \frac{\partial h}{\partial x}(t,z)(x-\hat x)\\
    & \leq h(t, \hat x) +\left\|\frac{\partial h}{\partial x}(t,z)\right\|\|(x-\hat x)\|\leq  h(t, \hat x) +l\epsilon.
\end{align*}
Thus, $h(t, \hat x)\leq -l\epsilon$ implies that $h(t,x)\leq 0$. Note that the time derivative of $\hat h$ along the trajectories of \eqref{eq: NL pert hat x} reads
\begin{align*}
    \dot {\hat h}(t,\hat x) & = L_{f}h(t,\hat x)+L_{g}h(t,\hat x)u+L_{d}h(t,\hat x)+\frac{\partial h}{\partial t}(t,\hat x)\\
    & \leq L_{f}h(t,\hat x)+L_{g}h(t,\hat x)u+\left\|\frac{\partial h}{\partial x}d(t,\hat x)\right\|+\frac{\partial h}{\partial t}(t,\hat x)\\
    & \leq L_{f}h(t,\hat x)+L_{g}h(t,\hat x)u+\frac{\partial h}{\partial t}(t,\hat x) + l\gamma\\
    & \overset{\eqref{eq: robust safety cond}}{\leq} \alpha(-h(t,\hat x)-l\epsilon) = \alpha(-\hat h(t,\hat x)).
\end{align*}
Thus, using Lemma \ref{lemma: nec suff safety}, we have that $\hat h(t,\hat x(t))\leq 0$ (or, $h(t,\hat x(t))\leq -l\epsilon)$ for all $t\geq 0$, i.e., the set $\hat S_\epsilon(t)$ is forward invariant for $\hat x(t)$ for all $\hat x(0)\in \hat S_\epsilon(0)$. Thus, we have $h(t,x)\leq 0$ for all $t\geq 0$, implying forward invariance of set $S(t)$ for all $\hat x(0)\in \hat S_\epsilon(0)$. 
\end{proof}

Thus, we can use the condition \eqref{eq: robust safety cond} to satisfy the safety requirements (ii)-(iii) in Problem \ref{Problem: multiagent ST}. Intuitively, Lemma \ref{lemma: robust CBF} guarantees that if $\hat x(t)\in S_\epsilon(t)$, then $x(t)\in S(t)$ for any $t\geq 0$, starting from which, forward invariance of the set $S(t)$ can be guaranteed. 

% {\color{blue}
% \begin{Remark}
% It is worth noting that while the disturbance $d$ and the sensor uncertainties $(\hat x-x)$ act very differently on the system, their treatment in the robust safety is very similar, i.e., both are treated by shrinking the safe set by a constant proportional to the worst-case bound of the disturbance. For the disturbance $\|d\|\leq \gamma$, Lemma \ref{lemma: robust CBF} requires the safe set $S$ to be shrunken by $l\gamma$, while for state-error $\|\hat x-x\|\leq \epsilon$, by $l\epsilon$. 
% \end{Remark}
% }

\begin{Remark}
Note that for the robust CBF condition, if the set $\hat S_\epsilon(0)$ is empty, then there exists no initial condition for which forward invariance of the set $S$ can be guaranteed based on Lemma \ref{lemma: robust CBF}.
\end{Remark}

\noindent Next, we present a robust CLF condition to guarantee FxTS of the closed-loop trajectories to the goal set. Consider a continuously differentiable function $V:\mathbb R^n\rightarrow\mathbb R$ with Lipschitz constant $l_V$. Using the mean-value theorem, we obtain
\begin{align}\label{eq: V hat V rel}
    V(x) & \leq V(\hat x) + l_V\epsilon, \; \forall x, \hat x\in \mathbb R^n,
\end{align}
from which we obtain that if $V(\hat x) \leq -l_V\epsilon$, then $V(x)\leq 0$. Using this and inspired from \cite[Definition 2]{garg2019prescribedTAC}, we define the notion of robust fixed-time CLF (FxT-CLF). 

\begin{Definition}[\textbf{Robust FxT-CLF-$S_G$}]
A continuously differentiable function $V:\mathbb R^n\rightarrow \mathbb R$ is called a robust FxT-CLF-$S_G$ for a set $S_G$ for \eqref{eq: NL pert cont affine} w.r.t. disturbance $d$ satisfying Assumption \ref{assum d} if $V(x)<0$ for $x\in \textnormal{int}(S_G)$, and there exists $\alpha \in \mathcal K_\infty$ such that $V(x)\geq \alpha(\|x\|_{S_G})$ for all $x\notin S_G$, satisfying
\begin{align}
     \inf_{u\in \mathcal U}\{L_f V(x)+L_g V(x)u\} \leq & -\alpha_1(V(x))^{\gamma_1}-\alpha_2(V(x))^{\gamma_2}\nonumber\\
     & +\delta_1(V(x))-l_V\gamma,\label{eq: dot V new ineq}
\end{align}
for all $x\notin S_G$ with $\alpha_1, \alpha_2>0$, $\delta_1\in \mathbb R$, $\gamma_1 = 1+\frac{1}{\mu}$, $\gamma_2 = 1-\frac{1}{\mu}$ for some $\mu>1$. 
\end{Definition}

Based on this, we can state the following result.

\begin{Lemma}\label{lemma: robust CLF}
Let $V:\mathbb R^n\rightarrow \mathbb R$ be such that $\hat V$ is a robust FxT-CLF-$S_G$ for \eqref{eq: NL pert hat x}. Then, there exists $u\in \mathcal U$, and a neighborhood $D$ of the set $S_G$ such that for all $\hat x(0)\in D$, the closed-loop trajectories of \eqref{eq: NL pert cont affine} reach the goal set $S_{G}$ in a fixed time $T$.
\end{Lemma}
\begin{proof}
Define $\hat V(\hat x)$ per \eqref{eq: hat func def} so that $\dot{\hat V}(\hat x) = \dot V(\hat x)$. Note that \eqref{eq: dot V new ineq} implies that there exists $u\in \mathcal U$ such that
\begin{align*}
    \dot{\hat V}(\hat x) & = L_f{V}(\hat x)+L_g{V}(\hat x)u + L_d{V}(\hat x) \\
    & \leq -\alpha_1{\hat V}(\hat x)^{\gamma_1}-\alpha_2{\hat V}(\hat x)^{\gamma_2}+\delta_1{\hat V}(\hat x).
\end{align*}
Thus, from \cite[Theorem 1]{garg2021FxTSDomain}, we obtain that there exists a domain $D$ and fixed time $0<T<\infty$ (that are functions of $\frac{\delta_1}{2\sqrt{\alpha_1\alpha_2}}$) such that $\hat V(\hat x(T)) = 0$ for all $\hat x(0)\in D$. Thus, we obtain that $V(\hat x(T))\leq -l_V\epsilon$, which, in light of \eqref{eq: V hat V rel}, implies $V(x(T)) \leq 0$ for all $\hat x(0)\in D$. Since $V(x)\geq \alpha(\|x\|_{S_G})$, $V(x(T)) \leq 0$ implies that $x(T) \in S_G$, which completes the proof.
\end{proof}

\noindent The robust FxT-CLF condition guarantees that if the estimated state $\hat x$ reaches a certain level set in the interior of the set $S_G$, quantitatively given as $\{\hat x\; |\; V(\hat x)\leq -l_V\epsilon\}$, then the actual state $x$ reach the zero sub-level set of $V$, and thus, reach the set $S_G$.

\begin{Remark}
For Lemma \ref{lemma: robust CLF}, it is required that the set $\{\hat x\; |\; h_G(\hat x)\leq -l_G\epsilon\} \neq \emptyset$. Otherwise, if the minimum value of the function $h_G$ exceeds $-l_G\epsilon$, i.e., $h_{G,min} \triangleq \min\limits_{x\in S_G}h_G(x)>-l_G\epsilon$, so that $\{\hat x\; |\; h_G(\hat x)\leq -l_G\epsilon\} = \emptyset$, it is not possible for $\hat h_G(\hat x)$ to go to zero. In such cases, \eqref{eq: dot V new ineq} implies that the closed-loop trajectories only reach the set $\{x\; |\; h_G(x)\leq h_{G,min}+l_G\epsilon\}$, leading to input-to-state stability. In this work, we assume that $\{\hat x\; |\; h_G(\hat x)\leq -l_G\epsilon\} \neq \emptyset$.
% One such example is the case when the goal set is a singleton, i.e., $S_G = \{x\; |\; \|x-x_G\|\leq 0\} = \{x_G\}$ for some $x_G\in \mathbb R^n$; in this case $\textnormal{int}(S_G)$ is empty and $\min\limits_{x\in S_G}h_G(x) = 0$, and thus, condition \eqref{eq: dot V new ineq} only guarantees that the closed-loop trajectories reach the set $\{x\; |\; \|x-x_G\|\leq l_G\epsilon\}$. 
\end{Remark}

With robust CBF and robust FxT-CLF conditions at hand, we can determine whether a given control input can render a safe set forward invariant, and drive the closed-loop trajectories to the desired goal set in the presence of disturbances and state-estimation error. Next, we address the problem of finding such a control input that satisfies the robust CBF and robust FxT-CLF conditions simultaneously, along with the input constraints. To this end, we resort to the QP-based method similar to \cite{garg2019prescribedTAC}, where the CBF and FxT-CLF conditions are cast as linear inequality constraints in a min-norm control problem. 

\vspace{2pt}
% \subsection{QP formulation}
\noindent \textbf{QP formulation}: \normalsize
Now, we discuss how to incorporate robust FxT-CLF and CBF constraints in a QP formulation, and discuss its feasibility. We use the result of Lemma \ref{lemma: robust CBF} to formulate robust CBF constraints for the sets $S_{S}$ and $S_{T}$, and Lemma \ref{lemma: robust CLF} to formulate the robust FxT-CLF-$S_G$ constraint for the goal set $S_G$. For the sake of brevity, we omit the arguments $\hat x$ and $(t,\hat x)$. Define $z = \begin{bmatrix}v^T & \delta_{1} & \delta_{2}& \delta_{3}\end{bmatrix}^T\in \mathbb R^{m+3}$, and consider the following optimization problem{\small{
\begin{subequations}\label{QP gen ideal}
\begin{align}
\min_{z\in \mathbb R^{m+3}} \;&  \frac{1}{2}z^THz  + F^Tz\\
    \textrm{s.t.} \quad \quad  A_{u}v  \leq & \; b_{u}, \label{C1 cont const ideal}\\
    L_{f}\hat h_{G} + L_{g}\hat h_{G}v  \leq & \; \delta_{1}\hat h_{G}-\alpha_{1}\max\{0,\hat h_{G}\}^{\gamma_{1}} \nonumber\\
    & -\alpha_{2}\max\{0,\hat h_{G}\}^{\gamma_{2}}-l_{G}\gamma \label{C2 stab const ideal}\\
    L_{f}\hat h_S + L_{g}\hat h_Sv \leq &-\delta_{2}\hat h_S-l_S\gamma,\label{C3 safe const ideal}\\
    L_{f}\hat h_{T} + L_{g}\hat h_{T}v \leq &-\delta_{3}\hat h_{T}-\frac{\partial \hat h_{T}}{\partial t}-l_T\gamma,\label{C4 tv safe const ideal}
\end{align}
\end{subequations}}}\normalsize
where $H = \textrm{diag}\{\{w_{u_l}\}, w_1, w_2, w_3\}$ is a diagonal matrix consisting of positive weights $w_{u_l}, w_1, w_2,
w_3>0$ for $l = 1, 2, \dots, m$, $F = \begin{bmatrix}\mathbf 0_m^T & q & 0 & 0\end{bmatrix}^T$ with $q>0$ and functions $\hat h_G,\hat h_S$ (respectively, $\hat h_T$) are functions of $\hat x$ (respectively, $(t,\hat x)$) defined as per \eqref{eq: hat func def}. The parameters $\alpha_{1}, \alpha_{2}, \gamma_{1}, \gamma_{2}$ are fixed, and are chosen as $\alpha_{1} = \alpha_{2} = \mu\pi/(2\bar T)$, $\gamma_{1} = 1+\frac{1}{\mu}$ and $\gamma_{2} = 1-\frac{1}{\mu}$ with $\mu>1$ and $\bar T$ the user-defined time in Problem \ref{Problem: multiagent ST}. Define $\hat S_G = \{\hat x\; |\; h_G(\hat x)\leq -L_G\gamma\}$ so that $\hat x\in \hat S_G\implies x\in S_G$. 
% Below we explain what each constraint in the QP \eqref{QP gen ideal} encodes:
% \begin{itemize}
%     \item \eqref{C1 cont const ideal}: control input constraint;
%     \item \eqref{C2 stab const ideal}: convergence of the closed-loop trajectories to the goal set $S_{G}$ within the user-defined time $\bar T$ (Lemma \ref{lemma: robust CLF})\footnote{As shown in \cite[Corollary 1]{garg2019control}, the $\max$ function in \eqref{C2 stab const ideal} guarantees that once the closed-loop trajectories reach the goal set $S_G$, it is forward-invariant.};
%     \item \eqref{C3 safe const ideal} and \eqref{C4 tv safe const ideal}: forward-invariance of the safe sets $S_{S}$ and $S_{T}$, respectively, (Lemma \ref{lemma: robust CBF}).
% \end{itemize}
We are now ready to present our main result. Let the solution of \eqref{QP gen ideal} be denoted as $z^*(\cdot) = \begin{bmatrix}v^*(\cdot)^T & \delta_{1}^*(\cdot) & \delta_{2}^*(\cdot)& 
% \{\delta_{ij_c}\}&
\delta_{3}^*(\cdot)\end{bmatrix}^T$.

\begin{Theorem}\label{Thm ideal case main result}
The following holds for each agent $i$:
\begin{itemize}
    \item [(i)] The QP \eqref{QP gen ideal} is feasible for all $\hat x(t)\in \left(\textnormal{int}\left(\hat S_{S}\right)\cap \textnormal{int}\left(\hat S_{T}(t)\right)\right)\setminus \hat S_{G}$ for all $t\geq 0$;
    % \item[(ii)] If the solution $z^*$ is continuous in its arguments, then for all $\hat x(0)\in  \left(\mathcal D\cap\textnormal{int}(\hat S_{S})\cap \textnormal{int}(\hat S_{T}(0))\right)$, the trajectories of \eqref{eq: NL pert cont affine} under the effect of control input $u = v^*$ satisfy $h_{S}(x(t))<0$ and $h_{T}(t,x(t))<0$ for all $t\geq 0$ such that $\hat x(t)\notin \hat S_{G}$; 
    \item[(ii)] If the solution $z^*$ is continuous in its arguments and $\max\limits_{0\leq \tau\leq \bar T}\delta_{1}(x(\tau))\leq 0$, then the control input defined as $u = v^*$ guarantees convergence of the closed-loop trajectories to the goal set $S_{G}$ within time $\bar T$, i.e., the control input $u = v^*$ solves Problem \ref{Problem: multiagent ST} for all $\hat x(0)\in  \left(\mathcal D\cap\textnormal{int}(\hat S_{S})\cap \textnormal{int}(\hat S_{T}(0))\right)$.
\end{itemize}
\end{Theorem}

\begin{proof}
Part (i): Since $\hat x(t)\in (\textnormal{int}(\hat S_{S}\cap \textnormal{int}(\hat S_{T}(t)))\setminus S_{G}$, we have that $\hat h_{S}(\hat x), \hat h_{T}(t,\hat x), \hat h_{G}(\hat x) \neq 0$ for all $t\geq 0$. Choose any $v = \bar v\in\mathcal U$ and define $\delta_{1} = \frac{L_{f}\hat h_{G} + L_{g}\hat h_{G}\bar v+\alpha_{1}\hat h_{G}^{\gamma_{1}}+\alpha_{2}\hat h_{G}^{\gamma_{2}}+l_V\gamma}{\hat h_{G}},$ which is well-defined for all $\hat x\notin \hat S_{G}$, so that \eqref{C2 stab const ideal} is satisfied with equality. Similarly, we can define $\bar \delta_{2}, \bar \delta_{3}$ so that \eqref{C3 safe const ideal}-\eqref{C4 tv safe const ideal} are satisfied with equality. Thus, there exists $\bar z = \begin{bmatrix}\bar v^T & \bar \delta_{1} & \bar \delta_{2}& \bar \delta_{3}\end{bmatrix}^T$ such that all the constraints of QP \eqref{QP gen ideal} are satisfied. 

% Part (ii): Consider any $0<T<\infty$ and define $\Delta_i = \sup\limits_{0\leq t\leq T}|\delta_i^*(\hat x(t))|$ for $i = 2, 3$. Since the solution $z^*$ is a continuous function, $\delta_2^*, \delta_3^*$ are also continuous, and thus, $\Delta_2, \Delta_3$ are well-defined. Using this, and the fact that $v^*$ satisfies the constraints \eqref{C3 safe const ideal}-\eqref{C4 tv safe const ideal}, we obtain that \eqref{eq h dot class K func} is satisfied for $h_{S}(\hat x)+l_S\gamma$ and $h_{T}(t,\hat x)+l_T\gamma$ with $\alpha(r) = \Delta_2r $ and $\alpha(r) = \Delta_3r$, respectively. Thus, using Lemma \ref{lemma forw var int set}, we obtain that $h_{S}(\hat x(t))<-l_S\gamma$ and $h_{T}(t,\hat x(t))<-l_T\gamma$ for all $t\geq 0$, which in turn implies $h_{S}( x(t))<0$ and $h_{T}(t,x(t))<0$ for all $t\geq 0$ such that $\hat x(t)\notin \hat S_G$.

Part (ii): The condition \eqref{C2 stab const ideal} implies that the function $\hat h_G$ is a robust FxT-CLF-$S_G$ for \eqref{eq: NL pert hat x}. Thus, using Lemma \ref{lemma: robust CLF} and \cite[Theorem 1]{garg2021FxTSDomain}, we obtain that $\hat h_G(\hat x(t)) \leq 0$ for $t\geq \bar T$, which implies that $h_G(\hat x(t))\leq -l_G\gamma$, which in turn implies $h_G(x(t))\leq 0$ for $t\geq \bar T$ for all $\hat x(0)\in \mathcal D\cap\textnormal{int}(\hat S_{S})\cap \textnormal{int}(\hat S_{T}(0))$. Furthermore, conditions \eqref{C3 safe const ideal} and \eqref{C4 tv safe const ideal} imply that the functions $\hat h_S$ and $\hat h_T$ are robust CBFs for \eqref{eq: NL pert hat x}, and thus, the set $S_S\cap S_T$ is forward-invariant for the closed-loop trajectories of \eqref{eq: NL pert cont affine}.
% Part (ii) guarantees forward-invariance of the safe sets $S_{S}$ and $S_{T}(t)$ for all $t\geq 0$. 
Thus, the control input $u = v^*$ solves Problem \ref{Problem: multiagent ST} for all $\hat x(0)\in D = \mathcal D\cap\textnormal{int}(\hat S_{S})\cap \textnormal{int}(\hat S_{T}(0))$.
\end{proof}

\noindent It is worth noting that the constraints in the QP \eqref{QP gen ideal} are a function of the estimated state $\hat x$, and not the actual state $x$, which is unknown. Thus, the resulting control input $u = v^*(\hat x)$ is realizable. Before presenting the case study, we provide some discussion on the main result. 

\begin{Remark}\label{remark: int set forw inv}
Theorem \ref{Thm ideal case main result} guarantees that starting from the intersection of the interiors of the safe sets, the closed-loop trajectories remain inside the interior of these sets. The case when the initial conditions lie on the intersection of the boundaries of the safe sets requires strong viability assumptions such as the existence of $u$ such that \eqref{eq: safety cond} holds for both $h_{S}$ and $h_{T}$ for all $x\in \partial S_{S}\cap \partial S_{T}$. 
% Under this condition, the result in Theorem \ref{Thm ideal case main result} can be extended to the set $D = \mathcal D\cap(\hat S_{S})\cap (\hat S_{T}(t))$.
\end{Remark}

\begin{Remark}\label{remark: continuity QP solution}
We impose continuity requirements on the solution of the QP \eqref{QP gen ideal} to use the traditional Nagumo's viability theorem to guarantee forward invariance of a set. Prior work e.g., \cite{ames2017control,glotfelter2017nonsmooth,garg2019prescribedTAC} discusses conditions under which the solution of parametric QP such as \eqref{QP gen ideal} is continuous, or even Lipschitz continuous. More recently, utilizing the concept of strong invariance and tools from non-smooth analysis, forward invariance of a set requiring that the control input is only measurable and locally bounded is discussed in \cite{usevitch2020strong}.  
\end{Remark}

\begin{Remark}
Note that the result in part (iii) of Theorem \ref{Thm ideal case main result} requires $\delta_{1}\leq 0$ so that the control input $u$ solves the convergence requirement of Problem \ref{Problem: multiagent ST}. When this condition does not hold, the closed-loop trajectories satisfy the safety requirements, but may not converge to the goal set from any arbitrary initial condition $x(0)\notin S_{G}$ (see \cite{garg2019prescribedTAC}). 
% Among other reasons, such as the goal set being ``too far away" from the initial condition that it is not possible to complete the task in the given time under the given input constraints, one of the reasons is that the agents might encounter a deadlock situation by getting stuck at the boundary of one or both of the safe sets. The authors in \cite{wang2017safety} characterize various types of deadlocks depending upon the value of the slack variables in the optimization problem and discuss methods of resolving some of the deadlock scenarios. While quite important from the practical point of view, this analysis is out of the scope of this paper and is left as an exercise for future investigation.   
\end{Remark}

\section{Case Study}\label{sec: simulations}
We consider a numerical case-study involving underactuated underwater autonomous vehicles with state $X_i\in \mathbb R^6$, modeled as{\footnotesize
\begin{align}\label{eq: NL dyn example}
\hspace{-10pt}
% \begin{split}
    % \dot x_i & = u_i\cos\phi_i - v_i\sin \phi_i,\\ 
    % \dot y_i & = u_i\sin\phi_i + v_i\cos \phi_i, \\
    % \dot \phi_i & = r_i,\\
    % m_{11}\dot u_i &= m_{22}v_ir_i + X_uu_i + X_{u|u|}|u_i|u_i+\tau_{u,i},\\
    % m_{22}\dot v_i &= -m_{11}u_ir_i + Y_vv_i + Y_{v|v|}|v_i|v_i,\\
    % m_{33}\dot u_i &= (m_{11}-m_{22})u_iv_i + N_rr_i + N_{r|r|}|r_i|r_i+\tau_{r,i},
 \begin{bmatrix}\dot x_i\\ \dot y_i\\ \dot \phi_i  \\ m_{11}\dot u_i \\  m_{22}\dot v_i \\m_{33}\dot r_i  \end{bmatrix} = \begin{bmatrix}u_i\cos\phi_i - v_i\sin \phi_i+ V_w\cos(\theta_w)\\ 
      u_i\sin\phi_i + v_i\cos \phi_i+ V_w\sin(\theta_w) \\
     r_i\\
   m_{22}v_ir_i + X_uu_i + X_{u|u|}|u_i|u_i\\
   -m_{11}u_ir_i + Y_vv_i + Y_{v|v|}|v_i|v_i\\
    (m_{11}-m_{22})u_iv_i + N_rr_i + N_{r|r|}|r_i|r_i\end{bmatrix}+\begin{bmatrix}0\\ 
      0\\
     0\\
   \tau_{u,i}\\
  0\\
   \tau_{r,i}\end{bmatrix}
% \end{split}
\end{align}}\normalsize
where $z_i = [x_i,\;y_i,\;\phi_i]^T$ is the configuration vector of the $i$-th agent, $[u_i,\;v_i,\;r_i]^T$ are the velocities (linear and angular) w.r.t the body-fixed frame, $\tau_i = [\tau_{u,i}, \; \tau_{r,i}]^T$ is the control input vector where $\tau_{r,i}$ are the control input along the surge ($x$-axis) and yaw degree of freedom, respectively, $X_u, Y_v, N_r$ are the linear drag terms, and $X_{u|u|}, Y_{v|v|}, N_{r|r|}$ are the non-linear drag terms (see \cite{panagou2014dynamic} for more details). The additive disturbance $d = \begin{bmatrix}V_w(X_i, t)\cos(\theta_w(X_i, t))\\ V_w(X_i,t)\sin(\theta_w(X_i, t))\end{bmatrix}$ with $|V_w(X_i, t)|\leq \gamma$ models the effect of an unknown, time-varying water current acting on the system dynamics of each agent. We also consider measurement uncertainties in the state estimates as stated in Assumption \ref{assum hat x eps}. The system dynamics is under-actuated since there is no control input in the sway degree of freedom ($y$-axis). The multi-task problem considered for the case study is as follows (see Figure \ref{fig: fov}):
\begin{Problem*}
Compute $\tau_i\in \mathcal U_i = [-\tau_{u,m}, \tau_{u,m}]\times[-\tau_{r,m}, \tau_{r,m}]$, $\tau_{u,m},\tau_{r,m}>0$, such that each agent
\begin{itemize}
    \item[(i)] Reaches an assigned goal region around a point $g_i\in \mathbb R^2$ within a user-defined time $T$;
    % , i.e., $\begin{bmatrix}x_i(t) & y_i(t)\end{bmatrix}^T\to g_i$ as $t\to T$;
    \item[(ii)] Keeps their respective point-of-interest $p_i\in \mathbb R^2$ in their field of view (given as a sector of radius $R>0$ and angle $\alpha>0$);
    % , i.e., $z_i(t)\in \mathcal F = \{z\in \mathbb R^3\; |\; \|\begin{bmatrix}x_i & y_i\end{bmatrix}^T-p_i\| \leq R, |\angle\left(p_i-\begin{bmatrix}x_i & y_i\end{bmatrix}^T\right)-\phi_i| \leq \alpha \}$ for all $t\geq 0$;
    \item[(iii)] Maintains a safe distance $d_s$ w.r.t. other agents;
    % , i.e., $\|\begin{bmatrix}x_i(t) & y_i(t)\end{bmatrix}^T-\begin{bmatrix}x_j(t) & y_j(t)\end{bmatrix}^T\|\geq d_s$ for all $t\geq 0$, $i\neq j$,
\end{itemize}
where $\angle(\cdot)$ is the angle of the vector $(\cdot)$ with respect to the $x$-axis of the global frame. 
\end{Problem*}

\begin{figure}[t]
    \centering
        \includegraphics[width=0.85\columnwidth,clip]{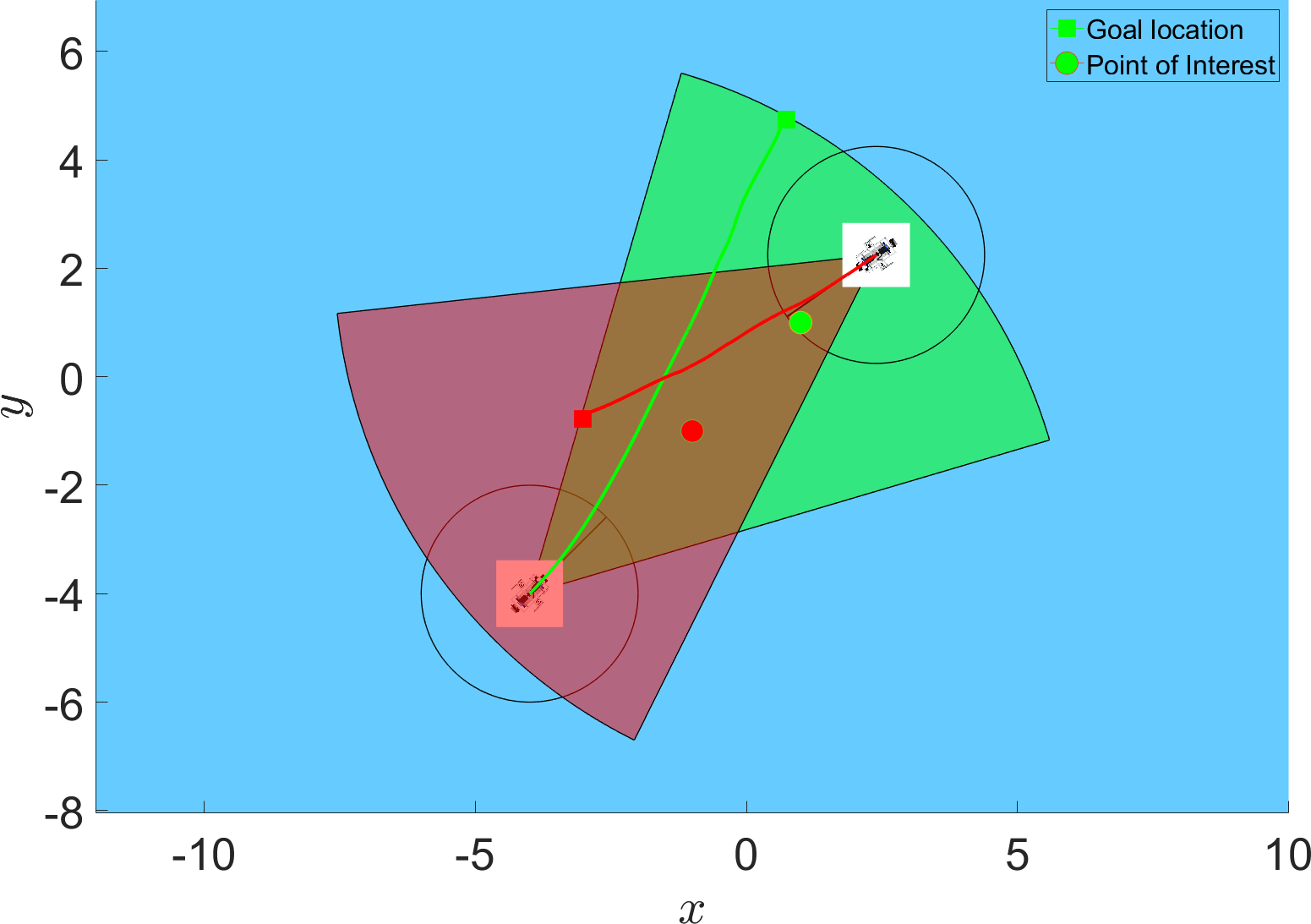}
    \caption{Problem setting for the two-agent case.}\label{fig: fov}
\end{figure}

\begin{table}[b]
\caption{Dynamic parameters as taken from \cite{panagou2014dynamic}.}
\centering
\begin{tabular}{|l|l|l|l|l|l|}
\hline
$m_{11}$ & 5.5404 & $X_u$ & -2.3015 & $X_{u|u|}$ & -8.2845 \\ \hline
$m_{22}$ & 9.6572 & $Y_v$ & -8.0149 & $Y_{v|v|}$ & -23.689 \\ \hline
$m_{22}$ & 1536   & $N_r$ & -0.0048 & $N_{r|r|}$ & -0.0089 \\ \hline
\end{tabular}\label{tab:1}
\end{table}

Note that (ii) requires safety with respect to a static safe set, while requires (iii) safety with respect to a time-varying safe set. The parameters used in the case study are given in Table \ref{tab:1}.  First, we construct CLF and CBFs to guarantee convergence to the desired location, and invariance of the required safe sets, respectively. Consider the function $ h_{ij} = d_s^2-\left\|\begin{bmatrix}x_i(t) & y_i(t)\end{bmatrix}^T-\begin{bmatrix}x_j(t) & y_j(t)\end{bmatrix}^T\right\|^2$, defined for $i\neq j$, so that $h_{ij}\leq 0$ implies that the agents maintain the safe distance $d_s$. Since the function $h_{ij}$ is relative degree two function with respect to the dynamics \eqref{eq: NL dyn example}, we use the second order safety condition discussed in \cite{wang2017safety}. Similarly, for keeping the point-of-interest in the field of view, we use two separate CBFs, defined as $h_\phi  = \left|\angle\left(p_i-\begin{bmatrix}x_i & y_i\end{bmatrix}^T\right)-\phi_i\right|^2-\alpha^2, h_R  = \left\|\begin{bmatrix}x_i & y_i\end{bmatrix}^T-p_i\right\|^2 - R^2$, so that $h_\phi(z_i)\leq 0, h_R(z_i)\leq 0$ implies that $z_i\in \mathcal F$. For $h_\phi, h_R$, we use the relative degree 2 condition \eqref{eq: h rel deg 2}. Finally, we define the CLF as $V = \frac{1}{2}(X_i-X_{di})^T(X_i-X_{di})$, where $X_i\in \mathbb R^6$ is the state vector of the $i$-th agent, and $X_{di}\in \mathbb R^6$ its desired state, defined as $ X_{di} = \begin{bmatrix}g_i\\ \theta_g\\ c_1\left\|g_i-\begin{bmatrix}x_i & y_i\end{bmatrix}^T\right\|\cos(\theta_g-\phi_i)\\ c_1\left\|g_i-\begin{bmatrix}x_i & y_i\end{bmatrix}^T\right\|\sin(\theta_g-\phi_i)\\ c_2(\theta_g-\phi_i)\end{bmatrix}$,
where $\theta_g = \angle\left(g_i-\begin{bmatrix}x_i & y_i\end{bmatrix}^T\right)$ and $c_1, c_2>0$ are some constants. We consider 4 agents for the numerical simulations.  First, we fix $\gamma = 0.5$ and $\epsilon = 0.5$. Figure \ref{fig:4 agent traj} shows the path traced by 4 agents. The solid circular region represents the goal set defined as $\{X\; |\; V(X)\leq 0.1\}$, and the square boxes denote the point of interests $p_i$ for each agent.\footnote{A video of the simulation is available at: \href{https://tinyurl.com/y32oa4p4}{https://tinyurl.com/y32oa4p4}.} Figure \ref{fig:4 agent V comp} plots $V_M = \max_i\{V_i\}$ showing the convergence of the agents to their respective goal sets, while satisfying all the safety constraints, as can be seen from Figure \ref{fig:h max}, which plots $h_M = \max\{h_{ij},h_R, h_\phi\}$ showing that all the CBFs are non-positive at all times for all the three cases. 

\begin{figure}[t]
    \centering
        \includegraphics[width=1\columnwidth,clip]{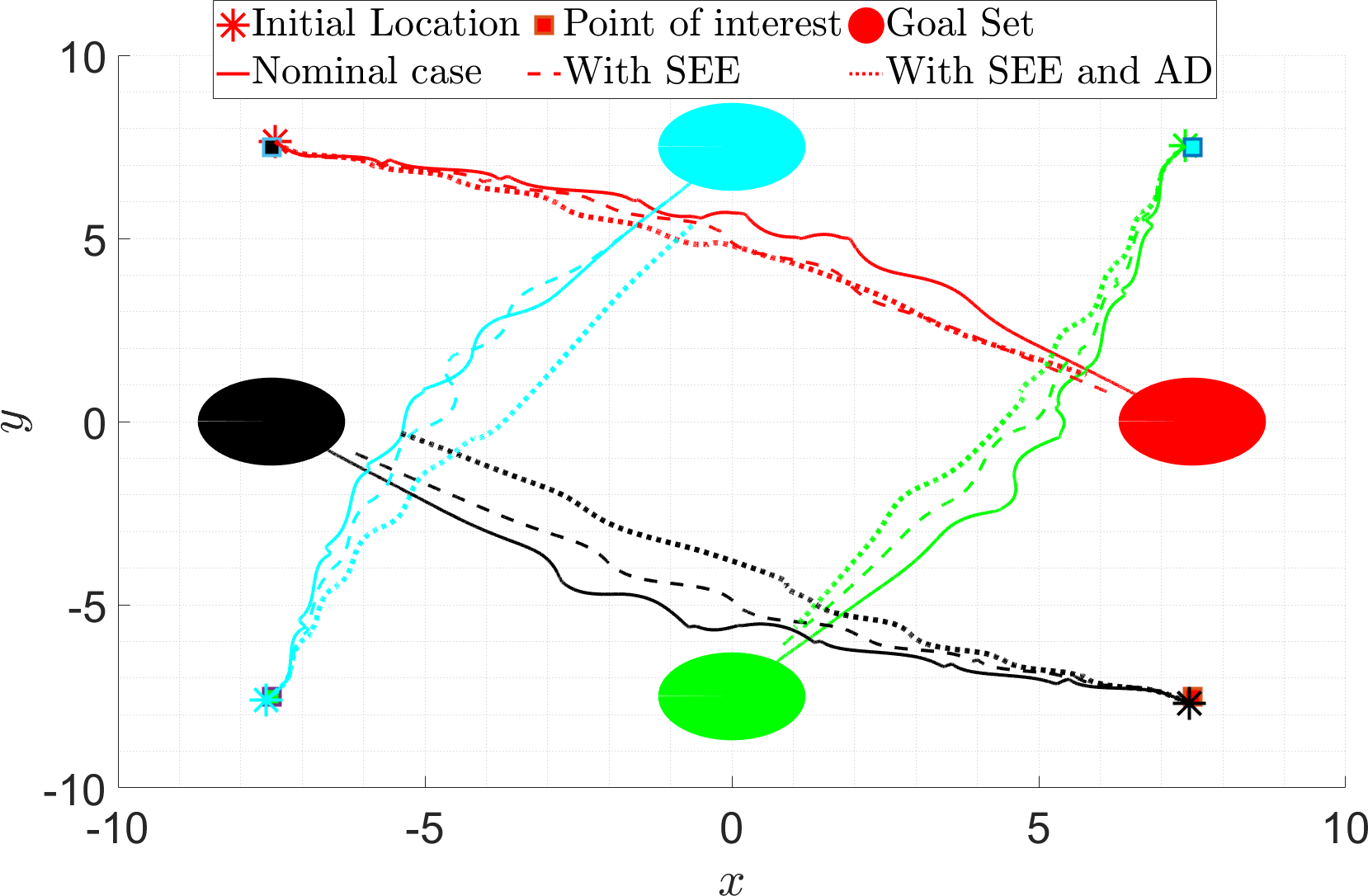}
    \caption{Closed-loop paths traced by agents in a 4 agents scenario for the nominal case, i.e., without any disturbance (solid lines), with only state estimation error (SEE) (dashed lines) and with both SEE and additive disturbance (AD) (dotted lines).}\label{fig:4 agent traj}
\end{figure}

Figures \ref{fig:tau u} and \ref{fig:tau r} show the control inputs $\tau_u$ and $\tau_r$, respectively, for the 4 agents, and it can be seen that the control input constraints are satisfied at all times. 

\begin{figure}[!ht]
    \centering
        \includegraphics[width=1\columnwidth,clip]{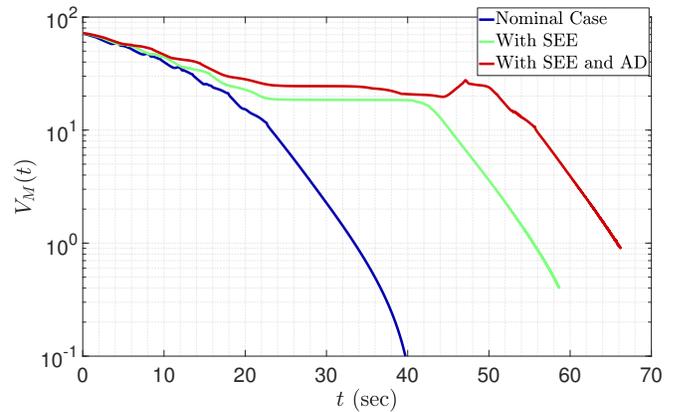}
    \caption{Point-wise maximum of Lyapunov functions $V_M(t) = \max_{i}\{V_i(t)\}$ with time for the three cases.}\label{fig:4 agent V comp}
\end{figure}

\begin{figure}[!ht]
    \centering
        \includegraphics[width=1\columnwidth,clip]{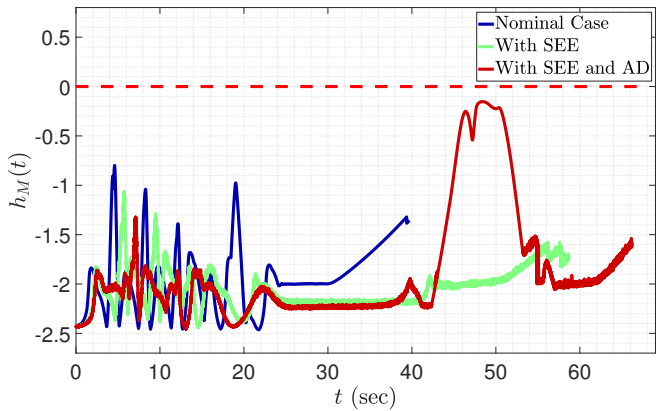}
    \caption{Point-wise maximum of CBFs $h_{ij},h_R, h_\phi$, showing satisfaction of all the safety constraints for the three cases.}\label{fig:h max}
\end{figure}

\begin{figure}[!ht]
    \centering
        \includegraphics[width=1\columnwidth,clip]{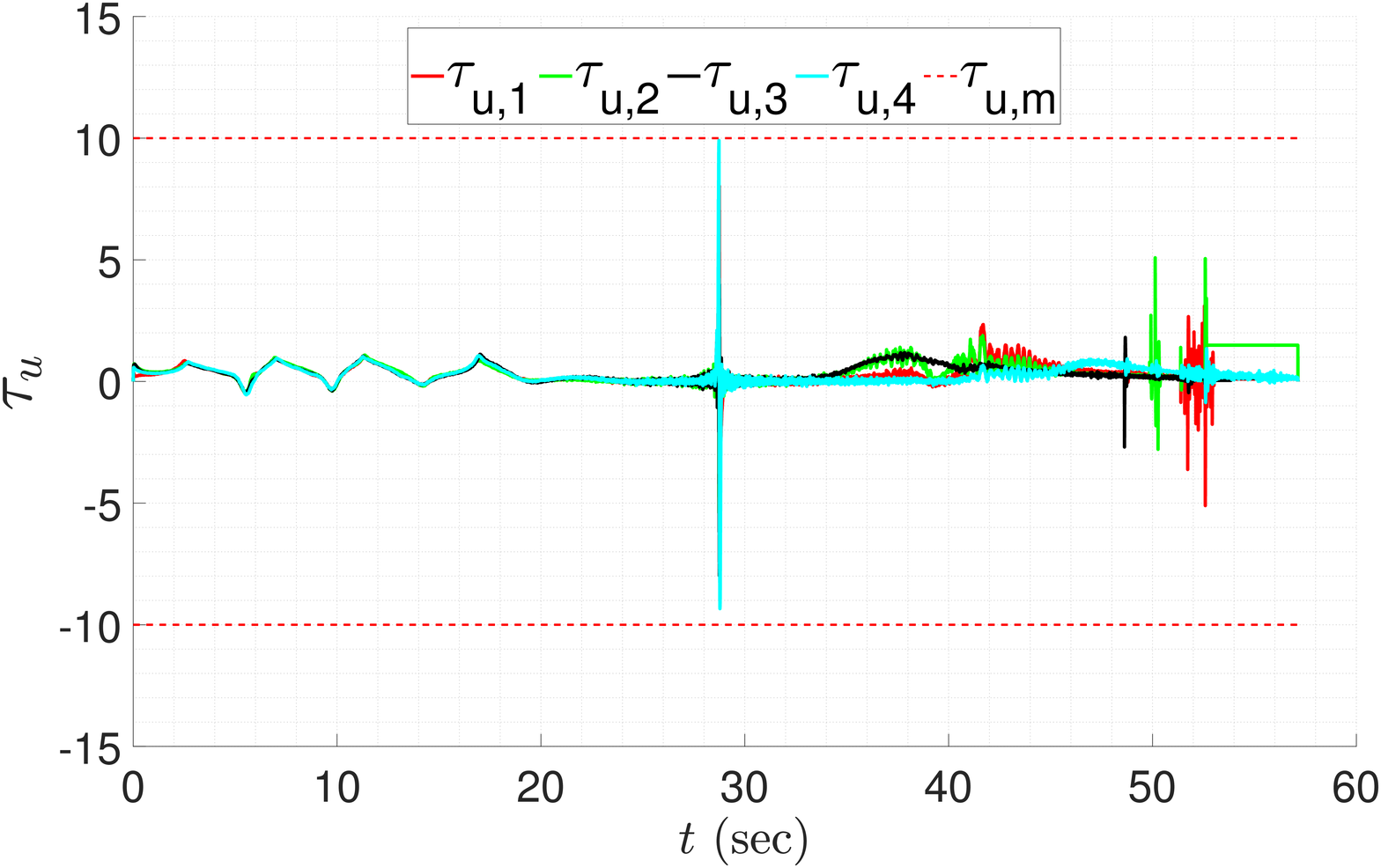}
    \caption{Control input $\tau_u$ for each agent. }\label{fig:tau u}
\end{figure}

\begin{figure}[!ht]
    \centering
        \includegraphics[width=1\columnwidth,clip]{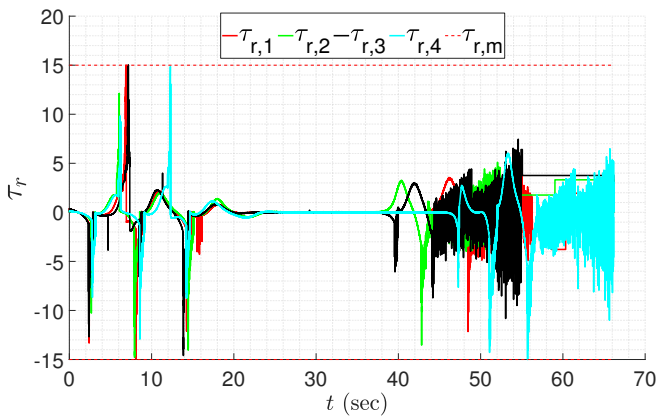}
    \caption{Control input $\tau_r$ for each agent. }\label{fig:tau r}
\end{figure}

% Figure \ref{fig:V d} plots the value of the Lyapunov function $V$ with time for various values of the upper-bound $\gamma$ between 0.5 and 5. The value of $\gamma$ increases from blue to red, and it can be observed from the figure that the final value of the function $V$ increases with $\gamma$. 

% \begin{figure}[!ht]
%     \centering
%         \includegraphics[width=1\columnwidth,clip]{V_dist_d.eps}
%     \caption{Lyapunov function $V$ with time for various disturbance bound $\gamma$ for a 2 agent scenario.}\label{fig:V d}
% \end{figure}

\section{Conclusion}
We considered a multi-task control synthesis problem for a class of nonlinear, control-affine systems under input constraints, where the objectives include remaining in a static safe set and a time-varying safe set and reaching a goal set within a fixed time. We also considered additive disturbances in the system dynamics and bounded state-estimation errors. We utilized robust CBFs to guarantee safety, and robust FxT-CLF to guarantee fixed-time reachability to given goal sets. Finally, we formulated a QP, incorporating safety and convergence constraints using slack variables so that its feasibility is guaranteed. We showed that under certain conditions, control input defined as the solution of the proposed QP solves the multi-task problem, even in the presence of the considered disturbances and input constraints.

One of the drawbacks of the presented method is conservatism due to the absence of knowledge of the structure of the disturbance. In the future, we would like to study online learning-based methods to learn estimates of the disturbances, so that the formulation can be made less conservative.
% The requirement of continuity of the solution of the QP can be relaxed by utilizing tools from the nonsmooth analysis. We would like to use these tools in the context of the problem considered in this paper and study the minimal set of requirements that can guarantee forward-invariance of safe sets in a distributed multiagent setting. 
% From a practical point of view, we would also like to consider other classes of disturbances, such as process noise in the system dynamics of each agent 

% \renewcommand{\baselinestretch}{0.9}

\bibliographystyle{IEEEtran}
\bibliography{myreferences}

\end{document}